\documentclass{amsart}

\usepackage{amsmath, amsthm, amssymb, amstext, amsfonts}
\usepackage{enumerate}
\usepackage{graphicx}
\usepackage[applemac]{inputenc}

\theoremstyle{plain}

\newtheorem{thm}{Theorem}[section]
\newtheorem{lem}[thm]{Lemma}
\newtheorem{cor}[thm]{Corollary}
\newtheorem{prop}[thm]{Proposition}
\newtheorem{rem}[thm]{Remark}

\theoremstyle{definition}

\newtheorem{exam}[]{Example}

\newcommand{\sm}{\ensuremath{\smallsetminus}}

\newcommand{\diam}{\textnormal{diam}}

\newcommand{\inv}{\ensuremath{^{-1}}}

\newcommand{\rand}{\partial}
\newcommand{\Aut}{\textnormal{Aut}}

\newcommand{\es}{\ensuremath{\emptyset}}

\newcommand{\sub}{\subseteq}

\newcommand{\sequ}[1]{(#1_i)_{i\in{\mathbb N}}}

\newcommand{\comment}[1]{}

\newcommand{\nat}{{\mathbb N}}
\newcommand{\real}{{\mathbb R}}

\newcommand{\ganz}{{\mathbb Z}}

\renewcommand{\L}{\ensuremath{\mathcal L}}
\renewcommand{\H}{\ensuremath{\mathcal H}}

\newenvironment{txteq*}
  {
    \begin{equation*}
    \begin{minipage}[c]{0.85\textwidth} 
    \em                                
  }
  {\end{minipage}\end{equation*}\ignorespacesafterend}

\begin{document}

\title[Group actions on metric spaces]{Group actions on metric spaces:\\fixed points and free subgroups}
\author{Matthias Hamann}
\address{Matthias Hamann, Department of Mathematics, University of Hamburg, Bundes\-stra\ss e~55, 20146 Hamburg, Germany}
\maketitle

\begin{abstract}
We look at group actions on metric spaces, particularly at group actions on geodesic hyperbolic spaces.
We classify the types of automorphisms on these spaces and prove several results about the density of the hyperbolic limit set of the group in the whole limit set of the group.
In the case of graphs, our theorems hold also when the graphs are not hyperbolic.
\end{abstract}

\section{Introduction}

In many situations, groups acting on some topological space offer the alternative between the existence of a free subgroup $\ganz \ast \ganz$ and the existence of a fixed point in the space under the action of the group.
For example, if the space is a proper geodesic hyperbolic space, then such results can be found in~\cite{ABCFLMSS,CoornDelPapa,GhHaSur,gromov,W-FixedSets}.
For the case of connected locally finite graphs we refer to~\cite{H-AutoAndEndo,J-FiniteFixedSets,EoG,W-FixedSets} for these results.

\smallskip

As mentioned, there is a vast literature on this topic for hyperbolic spaces, usually for the case of proper geodesic hyperbolic spaces.
Some of the fundamental properties of these extend to geodesic hyperbolic spaces that are not proper.
For example:
\begin{enumerate}[$\bullet$]
\item Every automorphism (i.e.\ self-isometry) of~$X$ is either elliptic, hyperbolic, or parabolic (Theorem~\ref{thm_W1});
\item a group $G$ of automorphisms fixes either a bounded subset of~$X$ or a unique limit point of~$G$ in~$\rand X$, or $X$ has precisely two limit points of~$G$, or $G$ contains two hyperbolic elements that freely generate a free subgroup (Theorem~\ref{thm_W3}).
\end{enumerate}
I~suspect that these two facts may be regarded as `known' by some experts, but since I have been unable to find proofs in the literature, I~have included proofs below.
In addition to the previous statements, we shall also prove the following results for a geodesic hyperbolic space $X$ with a group $G$ of automorphisms of~$X$:
\begin{enumerate}[$\bullet$]
\item The hyperbolic limit set of~$G$ is dense in the limit set of~$G$ (Theorem~\ref{thm_W2});
\item the hyperbolic limit set of~$G$ is bilaterally dense in the limit set of~$G$ if and only if either $X$ has precisely two limit points of~$G$ or $G$ contains two hyperbolic elements without a common fixed point (Theorem~\ref{thm_P5});
\item if the limit set of~$G$ is infinite, then it is a perfect set (Theorem~\ref{thm_P7}).
\end{enumerate}
(We refer to Section~\ref{sec_Def} for definitions.)
All these results are known to be true for proper geodesic hyperbolic spaces, cp.~\cite{P-BilatDenseness,W-FixedSets}.

\medskip

To avoid too many hyperbolic technicalities, we build up a general topological setting (\emph{contractive $G$-completions}), in which we prove our results.
This topological setting will extend Woess's contractive $G$-compactifications \cite{W-FixedSets} to spaces that need not be proper.
We shall show in Section~\ref{sec_Hyp} that geodesic hyperbolic spaces are contractive $G$-completions.

Another feature of contractive $G$-completions is that geodesic hyperbolic spaces are not their only application: connected graphs with a suitable notion of ends are further examples of these spaces (Section~\ref{sec_App:Graphs}).
Unfortunately, we cannot take all vertex ends (see Example~\ref{exam_vertexEnds}) but if we concentrate on those that contain no ray whose vertex set has an infinite bounded subset, then these ends form a boundary that satisfies our requirements (Theorem~\ref{thm_VertexEnds}).
Another possibility (Theorem~\ref{thm_MetricEnds}) is to take metric ends instead of vertex ends.
(We refer to Section~\ref{sec_App:Graphs} for the definitions of these notions for ends.)
If the graph is locally finite, then the results are known, see~\cite{P-BilatDenseness,W-FixedSets}.

\section{Contractive $G$-completions}\label{sec_Def}\label{sec_results}

Let $X$ be a metric space.
If $G$ is a group of automorphisms on~$X$, then we call a regular Hausdorff space $\hat{X}\supseteq X$ a \emph{$G$-completion}
\begin{enumerate}[(1)]
\item\label{item_Completion} if the inclusion $X\to\hat{X}$ is a homeomorphism and $X$ is open and dense in~$\hat{X}$,
\item\label{item_CompletionI} if every element of~$G$ extends to a homeomorphism of~$\hat{X}$,
\item\label{item_CompletionII}\label{item_CompletionIII} if for every sequence $\sequ{g}$ in~$G$ with $d(x,xg_i)\to\infty$ for $i\to\infty$ and for some $x\in X$ either the set $\{xg_i\mid i\in\nat\}$ has an accumulation point in~$\rand X$ or $\langle g_i\mid i\in\nat\rangle$ contains an automorphism $\gamma$ such that each of the sets $\{x\gamma^i\mid i\in\nat\}$ and $\{x\gamma^{-i}\mid i\in\nat\}$ converges to some point in~$\rand X$ and
\item\label{item_CompletionIIStronger} if for every sequence $\sequ{g}$ in~$G$ and for any $x\in X$ such that the set ${\{xg_i\mid i\in\nat\}}$ converges to some boundary point $\eta$ and the set $\{xg_i\inv\mid i\in\nat\}$ has no accumulation point in~$\rand X$ there is a sequence $(h_j)_{j\in\nat}$ in $\langle g_i\mid i\in\nat\rangle$ such that each of the sets $\{xh_j^i\mid i\in\nat\}$ and ${\{xh_j^{-i}\mid i\in\nat\}}$ converges to distinct boundary points $\eta_j$, $\mu_j\in\rand X$, respectively, and such that the sequence $(\eta_j)_{j\in\nat}$ converges to~$\eta$.
\end{enumerate}
\setcounter{equation}{4}

A~completion $\hat{X}$ of~$X$ is \emph{projective} if for all sequences $\sequ{x}, \sequ{y}$ in~$X$ such that $\sequ{x}$ converges to~$\eta\in\rand X$ and such that $d(x_i,y_i)\leq M$ for some $M<\infty$ also the sequence $\sequ{y}$ converges to~$\eta$.
A~$G$-completion $\hat{X}$ of~$X$ is \emph{contractive} if it is projective and if for all sequences $\sequ{g}$ in~$G$ with
\[
xg_n\to\eta\in\rand X \quad \text{and} \quad xg_n\inv\to\mu\in\rand X
\]
for some $x\in X$ the sequence $(yg_n)_{n\in\nat}$ \emph{converges uniformly} to~$\eta$ outside every neighbourhood of~$\mu$ in~$\hat{X}$, that is, that for any open neighbourhoods $U$ of~$\eta$ and $V$ of~$\mu$, there is an $n_0\in\nat$ such that $yg_n\in U$ for all $y\in \hat{X}\sm V$ and all $n\geq n_0$.

\begin{lem}\label{lem_NoLocalBP}
Let $\hat{X}$ be a projective $G$-completion.
No bounded sequence in~$X$ converges to any $\eta\in\rand X$.
\end{lem}

\begin{proof}
Let us suppose that we find an $\eta\in\rand X$ and a bounded sequence $(x_i)_{i\in\nat}$ that converges to~$\eta$.
Then any constant sequence $(x)_{i\in\nat}$ in~$X$ converges to~$\eta$ due to projectivity.
But this contradicts the fact that $\hat{X}$ is Hausdorff.
\end{proof}

\begin{lem}\label{lem_UVExistNicely}
Let $\hat{X}$ be a projective $G$-completion and let $\eta$ and $\mu$ be distinct elements of~$\rand X$.
For every open neighbourhood $U$ of~$\eta$ with $\mu\notin\overline{U}$, there exists an open neighbourhood $V$ of~$\mu$ with $d(U,V)>0$ and $\overline{U}\cap\overline{V}=\es$.

Furthermore, we may choose $V$ so that it does not contain a prescribed point $x\in X\sm U$.
\end{lem}

\begin{proof}
As $\hat{X}$ is regular, we find an open neighbourhood $V'\sub X\sm (U\cup\{x\})$ of~$\mu$ and an open neighbourhood $U'$ of $\overline{U}\cup\{x\}$ that are disjoint.
Projectivity gives us that any sequence within a fixed distance $M>0$ to~$U$ converges to a boundary point in~$\overline{U}$ and hence not to~$\mu$.
So $V=V'\sm \overline{B}_M(U)$ is open, still has $\mu$ as an accumulation point, and satisfies the other assertions.
\end{proof}

We call an automorphism $g\in G$ on~$X$
\begin{enumerate}[$\bullet$]
\item \emph{elliptic} if it fixes a bounded  non-empty subset of~$X$;
\item\emph{hyperbolic} if it is not elliptic and if it fixes precisely two boundary points $\eta,\mu\in\rand X$;
\item \emph{parabolic} if it is not elliptic and if it fixes precisely one boundary point $\eta\in\rand X$.
\end{enumerate}

\begin{thm}\label{thm_W1}
Let $\hat{X}$ be a contractive $G$-completion of a metric space~$X$.
Then each $g\in G$ is either elliptic, hyperbolic, or parabolic.

Furthermore, if $g$ is hyperbolic and fixes the two boundary points $\eta$ and $\mu$, then $xg^n\to\eta$ and $xg^{-n}\to\mu$ for all $x\in X$ or vice versa, and if $g$ is parabolic, then $\{xg^n\mid n\in\ganz\}$ has precisely one accumulation point, the point fixed by~$g$.
\end{thm}

\begin{rem}
Note that in general for a parabolic element $g$ the analogous converging property as for hyperbolic elements need not be true, that is, at the end of Section~\ref{sec_App:Graphs} we shall give an example of a contractive $G$-completion $X$ that has a boundary point $\eta$ and a point $x\in X$ with $xg^n\not\to\eta$.
Due to projectivity, this implies $yg^n\not\to\eta$ for every $y\in X$.
\end{rem}

\begin{proof}[Proof of Theorem~\ref{thm_W1}]
Let $g\in G$ and $x\in X$.
Let us assume that $g$~is not elliptic.
Then the set $\{d(xg^n,xg^m)\mid m,n\in\ganz\}$ is unbounded and hence, the same is true for $\{d(x,xg^n)\mid n\in\nat\}$.
So we conclude by~(\ref{item_CompletionIII}) that $A:=\{xg^n\mid n\in\nat\}$ has an accumulation point $\eta\in\rand X$ and $B:=\{xg^{-n}\mid n\in\nat\}$ has an accumulation point $\mu\in\rand X$.

Since the elements of~$G$ are homeomorphisms on~$\hat{X}$, we know by projectivity of~$\hat{X}$ that
\[
\eta g=(\lim x g^{n_i})g=\lim (xg)g^{n_i}=\eta
\]
where $(g^{n_i})_{i\in\nat}$ is a subsequence of~$(g^i)_{i\in\nat}$ such that $xg^{n_i}\to\eta$ for $i\to\infty$.
So we have $\eta g=\eta$ and, analogously, we also have $\mu g=\mu$.

Let $\rand A$, $\rand B$ be the accumulation points of~$A$, $B$ in~$\rand X$, respectively.
Then the sets $\rand A$ and $\rand B$ are non-empty closed subsets of~$\rand X$.
First, we show that each of the two sets $A$ and $B$ has precisely one accumulation point.
Let us suppose that there is a second accumulation point $\eta'$ of~$A$.
We have $\eta'g=\eta'$, too.
The sequence $(xg^{-n_i})_{i\in\nat}$ is unbounded because of $d(x,xg^n)=d(xg^{-n},x)$.
If $(xg^{-n_i})_{i\in\nat}$ has no accumulation point in~$\rand X$, then there is a $z\in\ganz$ such that $(xg^{kz})_{k\in\nat}$ and $(xg^{-kz})_{k\in\nat}$ converge in~$\hat{X}$.
But then we have $|\rand A|=1=|\rand B|$, as $\hat{X}$ is projective and as $\{d(xg^z,xg^{z+i})\mid 0\leq i\leq z\}$ is bounded.
So $(xg^{-n_i})_{i\in\nat}$ converges to~$\mu$, a contradiction.
Hence, $(xg^{-n_i})_{i\in\nat}$ has an accumulation point in~$\rand X$, say $\mu$.
Let us take an infinite subsequence of~$(n_i)_{i\in\nat}$ such that $(xg^{-n_i})_{i\in\nat}$ converges for this subsequence to~$\mu$.
We may assume that $(n_i)_{i\in\nat}$ itself is this subsequence.
If $\eta\neq\mu$, let~$U$ and~$V$ be open neighbourhoods of~$\eta$ and $\mu$, respectively, with $x\notin U\cup V$ and $\overline{U}\cap\overline{V}=\es$.
Due to contractivity, there exists $m\in\nat$ with $(X\sm V)g^{n_m}\sub U$ and we conclude inductively $xg^{\ell n_m}\in U$ for all $\ell\in\nat$.
Due to projectivity, every accumulation point of~$A$ lies in~$\overline{U}$ as $\{d(xg^{n_m},xg^{i+n_m})\mid 0\leq i\leq n_m\}$ is bounded and, conversely, every accumulation point of~$B$ lies in~$\overline{V}$.
If $\eta=\mu$, let $U=V$ be an open neighbourhood of~$\eta$ with $\eta'\notin\overline{U}$ and let $(y_i)_{i\in\nat}$ be a sequence in $X\sm U$ that converges to~$\eta'$.
As $\hat{X}$ is contractive, there is an $m\in\nat$ with $y_ig^{n_m}\in U$ for all $i\in\nat$.
But then we have $\eta'=\eta' g^{n_m}\in\overline{U}$, a contradiction.
This shows that $\eta$ is the unique element of~$\rand A$.
Analogously, we obtain that $\mu$ is the unique element of~$\rand B$.

\medskip

We will show that no boundary point is fixed by~$g$ but $\eta$ and~$\mu$.
Let us suppose that there is $\nu\in\rand X\sm\{\eta,\mu\}$ with $\nu g=\nu$.
As $\hat{X}$ is regular, we may take open neighbourhoods $U$ and~$V$ of~$\eta$ and~$\mu$, respectively, with $\overline{U}\cap\overline{V}=\es$ such that $\nu\notin\overline{U}\cup\overline{V}$.
Let $(x_i)_{i\in\nat}$ be a sequence in~$X\sm(U\cup V)$ converging to~$\nu$.
By contractivity, we find an $n\in\nat$ with $x_ig^n\in U$ for all $i\in\nat$.
Hence, we have $\nu=\nu g^n\in\overline{U}$, a contradiction.
Thus, there are at most two boundary points, $\eta$ and $\mu$, of~$X$ fixed by~$g$ and $g$ is either parabolic or hyperbolic.
If $g$ is parabolic, then we just showed that the set $\{xg^n\mid n\in\ganz\}$ has precisely one accumulation point, as we showed earlier $|\rand A|=1$, and the same is true for $\{yg^n\mid n\in\ganz\}$ by projectivity.

\medskip

So let us assume that $\eta$ and~$\mu$ are distinct, that is, that $g$ is hyperbolic.
We have to show the converging property of hyperbolic automorphisms.
Let us first show that $xg^n$ and $xg^{-n}$ for $n\in\nat$ converge to $\eta$ and $\mu$, respectively.
Therefore, we show that we can find a sequence $(n_i)_{i\in\nat}$ such that $xg^{n_i}$ converges to~$\eta$ and $xg^{-n_i}$ converges to~$\mu$.
Let us take an arbitrary sequence $(n_i)_{i\in\nat}$ such that $xg^{n_i}$ converges to~$\eta$.
Let us suppose that $\mu$ is no accumulation point of~$xg^{-n_i}$.
As $d(x,xg^{-n_i})$ is unbounded, we know by~(\ref{item_CompletionIII}) that there is an $n\in\nat$ such that $xg^{nk}\to\eta$ and $xg^{-nk}\to\mu$ for $k\to\infty$.
By projectivity, this holds also for $g$ instead of~$g^n$.
Now, let $y\in X$.
As $\hat{X}$ is projective and $d(x,y)=d(xg^n,yg^n)$ for all $n\in\nat$, also the sequence $(yg^i)_{i\in\nat}$ converges to~$\eta$ and the sequence $(yg^{-i})_{i\in\nat}$ converges to~$\mu$.
This shows the additional statement on hyperbolic automorphisms.
\end{proof}

Notice that due to Theorem~\ref{thm_W1}, the automorphism $\gamma$ mentioned in~(\ref{item_CompletionIII}) is either hyperbolic or parabolic and in~(\ref{item_CompletionIIStronger}), we find infinitely many hyperbolic automorphisms whose directions converge to~$\eta$.

For a hyperbolic element $g$, let the boundary point to which the sequence $(xg^n)_{n\in\nat}$ for $x\in X$ converges be the \emph{direction} of~$g$.
Note that this definition does not depend on the point $x$ by projectivity.
By $g^+$ we denote the direction of~$g$ and with $g^-$ the direction of~$g\inv$.
For parabolic elements, we denote by $g^+$ and $g^-$ the unique fixed boundary point.
For a contractive $G$-completion $\hat{X}$ of~$X$, let the \emph{limit set} $\L(G)$ of~$G$ be the set of accumulation points in~$\rand X$ of~$xG$ for any $x\in X$ and let the \emph{hyperbolic limit set} $\H(G)$ of~$G$ be the set of directions of hyperbolic elements.
Again, these sets do not depend on the choice of~$x$ due to projectivity.

\begin{lem}\label{lem_CertainHypExists}
Let $\hat{X}$ be a contractive $G$-completion of a metric space $X$, let $U$ and $V$ be non-empty open subsets of~$\hat{X}$ with $d(U,V)>0$, $\overline{U}\cap \overline{V}=\es$, and $\overline{U}\cup\overline{V}\neq\hat{X}$, and let $g\in G$.
If $(\hat{X}\sm V)g\sub U$, then $g$ is hyperbolic with $g^+\in \overline{U}$ and $g^-\in \overline{V}$.
\end{lem}

\begin{proof}
First, we notice that $\hat{X}\sm U\sub Vg$ and hence $(\hat{X}\sm U)g\inv\sub V$.
As $U$ and $V$ are disjoint, we obtain inductively that $(\hat{X}\sm V)g^n\sub U$ and $(\hat{X}\sm U)g^{-n}\sub V$ for all $n\geq 1$.
Since $X$ is dense in~$\hat{X}$ and $\overline{U}\cap\overline{V}\neq\hat{X}$, we find an $x\in X\sm(U\cup V)$.
Let us show that the orbit of~$x$ under $g$ is not bounded.
Indeed, as $g$ is an automorphism, the inequality
\[
d(x,xg^n)\geq (n-1)d(U,V)+d(x,U)
\]
holds and shows that $g$ is not elliptic.
Thus, $g$ is either parabolic or hyperbolic according to Theorem~\ref{thm_W1}.
Due to~(\ref{item_CompletionIII}), the set $\{xg^n\mid n\in\nat\}$ has an accumulation point, which lies in~$\overline{U}$, and $\{xg^{-n}\mid n\in\nat\}$ has an accumulation point, which lies in~$\overline{V}$.
According to Theorem~\ref{thm_W1}, the automorphism $g$ cannot be parabolic, so it must be hyperbolic and we have $g^+\in\overline{U}$ and $g^-\in\overline{V}$.
\end{proof}

\begin{thm}\label{thm_W2}
Let $\hat{X}$ be a contractive $G$-completion of a metric space $X$.
\begin{enumerate}[{\em (i)}]
\item\label{item_W2_2} If $\L(G)$ has at least two elements, then $\H(G)$ is dense in~$\L(G)$.
\item\label{item_W2_1} The set $\L(G)$ has either none, one, two, or infinitely many elements.
\item\label{item_W2_3} The set $\H(G)$ has either none, two, or infinitely many elements.
\end{enumerate}
\end{thm}

\begin{proof}
To prove (\ref{item_W2_2}), let $\eta,\mu\in\L(G)$ be distinct and let $x\in X$.
Then there are sequences $\sequ{g}$ and $\sequ{h}$ in~$G$ with $xg_i\to\eta$ and $xh_i\to\mu$.
We show that in any neighbourhood of~$\eta$ we find a direction of a hyperbolic element.

We may assume that $(xg_i\inv)_{i\in\nat}$ has at most one accumulation point: if it has more than one, then we take a subsequence of $(g_i)_{i\in\nat}$ such that $(xg_i\inv)_{i\in\nat}$ converges in~$\hat{X}$.
Thus, we find an open neighbourhood $U$ of~$\eta$ with $x\notin U$ and such that $\overline{U}$ does not contain any accumulation point of~$(xg_i\inv)_{i\in\nat}$.
If $(xg_i\inv)_{i\in\nat}$ has no accumulation point, then we find with condition (\ref{item_CompletionIIStronger}) a sequence $(f_j)_{j\in\nat}$ of hyperbolic automorphisms such that $f_j^+\to\eta$ for $j\to\infty$.
Thus, we may assume that $(xg_i\inv)_{i\in\nat}$ converges to $\nu\in\rand X$.

We distinguish several cases.
First, let us assume that $\nu\neq\eta$.
Due to Lemma~\ref{lem_UVExistNicely}, we find for $U$ an open neighbourhood $V$ of~$\nu$ with $\overline{U}\cap\overline{V}=\es$, with ${d(U,V)>0}$, and with $x\notin V$.
As $\hat{X}$ is contractive, there is an $n\in\nat$ with $(\hat{X}\sm V)g_i\sub U$ for all $i\geq n$.
According to Lemma~\ref{lem_CertainHypExists}, for all $i\geq n$, the automorphism $g_i$ is hyperbolic with $g_i^+\in \overline{U}$ and $g_i^-\in V$.
So we have found directions of hyperbolic automorphism arbitrarily close to~$\eta$.

In the situation that $\sequ{xh\inv}$ does not have $\mu$ as an accumulation point, an analogous proof as above gives us a direction of a hyperbolic automorphism ${f\in\langle h_i\mid i\in\nat\rangle}$ in every neighbourhood of~$\mu$.
If either $f^+=\eta$ or $f^-=\eta$, then $\eta$ itself is a direction of a hyperbolic element.
Hence, we may assume that $f^+\neq\eta$ and we may also assume that $f^+\neq\nu$ by taking $f\inv$ instead of~$f$.
Applying contractivity, we obtain that $f^+g_n\in U$ for all $n\geq n_0$ for some $n_0\in\nat$.
As $f^+g_n$ is the direction of the hyperbolic automorphism $g_nfg_n\inv$, we obtain the direction of a hyperbolic automorphism in~$U$, too.

Let us now assume that $\sequ{xg\inv}$ converges to~$\eta$ and that $\sequ{xh\inv}$ converges to~$\mu$.
As $\eta\neq\mu$, there are again open neighbourhoods $U$ and~$V$ of~$\eta$ and~$\mu$, respectively, with $x\notin\overline{U}\cup\overline{V}$, with $d(U,V)>0$, and with $\overline{U}\cap\overline{V}=\es$ due to Lemma~\ref{lem_UVExistNicely}.
By contractivity, we find an $n\in\nat$ such that
\[(\hat{X}\sm U)g_i\sub U\quad\text{ and }\quad (\hat{X}\sm V)h_i\sub V\]
for all $i\geq n$.
For $f:=h_ng_n$ this implies
\[(\hat{X}\sm V)f\sub V g_n\sub U.\]
By Lemma~\ref{lem_CertainHypExists}, the automorphism $f$ is hyperbolic with $f^+\in \overline{U}$ and $f^-\in \overline{V}$.
This finishes the last case and shows that $\H(G)$ is dense in~$\L(G)$ as soon as $\L(G)$ has at least two elements.

\medskip

For the proof of~(\ref{item_W2_1}) and (\ref{item_W2_3}), let us assume that $\L(G)$ contains at least three elements.
As $\H(G)$ is dense in $\L(G)$ according to~(\ref{item_W2_2}) and as $\hat{X}$ is Hausdorff, there are two hyperbolic automorphisms $g$ and~$h$ that do not fix the same two boundary points of~$X$.
Let $\eta\in\rand X$ with $\eta g=\eta$ and $\eta h\neq\eta$.
Then due to contractivity, the sequence $(\eta h^n)_{n\in\nat}$ converges to~$h^+$.
Hence, the set $\{\eta h^n\mid n\in\nat\}$ is infinite.
On the other hand, the boundary point $\eta h^n$ is fixed by $h^{-n} g h^n$ which is, as it is conjugated to a hyperbolic automorphism, also hyperbolic.
Hence $\H(G)$ and $\L(G)$ are infinite.
Since every hyperbolic automorphism fixes two boundary points, we also have $|\H(G)|\neq 1$.
\end{proof}

A group $G$ acts \emph{discontinuously} on a metric space $X$, if there is a non-empty open subset $O\sub X$ with $Og\cap O=\emptyset$ for all non-trivial elements $g$ of~$G$.

\begin{thm}\label{thm_W3}
Let $\hat{X}$ be a contractive $G$-completion of a metric space $X$.
Then one of the following cases holds:
\begin{enumerate}[{\em(i)}]
\item\label{item_W3_2} $G$ fixes a bounded subset of~$X$;
\item\label{item_W3_3} $G$ fixes a unique element of~$\L(G)$;
\item\label{item_W3_4} $\L(G)$ consists of precisely two elements;
\item\label{item_W3_1} $G$ contains two hyperbolic elements that have no common fixed point and that freely generate a free subgroup of~$G$ that contains aside from the identity only hyperbolic elements and that acts discontinuously on~$X$.
\end{enumerate}
\end{thm}

\begin{proof}
First, let us assume that $G$ does not contain any hyperbolic automorphism.
Then Theorem~\ref{thm_W2}\,(\ref{item_W2_2}) implies that $|\L(G)|\leq 1$.
If $|\L(G)|=1$, then the unique element of~$\L(G)$ has to be fixed by~$G$ which shows that (\ref{item_W3_3}) is true in this situation.
Thus, we may assume that $\L(G)$ is empty.
According to~(\ref{item_CompletionIII}), the set $xG$ must be bounded for any $x\in X$.
So~(\ref{item_W3_2}) holds.

\medskip

Let us now assume that $G$ contains a hyperbolic automorphism.
Then we have $|\L(G)|\geq 2$.
If $|\L(G)|=2$, then (\ref{item_W3_3}) holds.
So we may assume that $|\L(G)|\neq 2$.
Thus, we have $|\H(G)|>2$, since $\H(G)$ is dense in~$\L(G)$ due to Theorem~\ref{thm_W2}\,(\ref{item_W2_2}).
So $G$ contains more than one hyperbolic element.
We shall show that either (\ref{item_W3_3}) or (\ref{item_W3_1}) holds.

Let us first consider the case that each two hyperbolic automorphisms have a common fixed point.
Then we shall show the existence of a boundary point in~$\L(G)$ that is fixed by all elements of~$G$.
Suppose that no such fixed point exists.
Let $g\in G$ be hyperbolic.
As $G$ contains more than one hyperbolic element that have in total more than two distinct directions, we know that $\{g^+,g^-\}$ is not $G$-invariant.
For every $h\in G$, the automorphism $h\inv gh$ is hyperbolic.
As each two hyperbolic automorphisms have a common fixed point, either $g^+h=(h\inv gh)^+$ or $g^-h=(h\inv gh)^-$ lies in $\{g^+,g^-\}$, in particular, we have $\{g^+,g^-\}h\cap\{g^+,g^-\}\neq\es$.

Let us suppose that there are $h_1,h_2\in G$ with
\[\{g^+,g^-\}h_1\cap \{g^+,g^-\}=\{g^+\}\quad\text{ and }\quad \{g^+,g^-\}h_2\cap \{g^+,g^-\}=\{g^-\}.\]
The automorphisms $g_i:=h_i\inv gh_i$ for $i=1,2$ are hyperbolic and hence have a common fixed point $\eta$.
But this fixed point can be neither $g^+$ nor~$g^-$ by the choices of~$h_1$ and~$h_2$.
Let $U$ and $V$ be disjoint open neighbourhoods of~$g^+$ and $g^-$, respectively, such that none of them contains~$\eta$.
As $\hat{X}$ is contractive, there is an $n\in\nat$ such that $(\hat{X}\sm V)g^n\sub U$ and $(\hat{X}\sm U)g^{-n}\sub V$.
Again, the automorphisms $f_1:=g^{-n}g_1g^n$ and $f_2:=g^ng_2g^{-n}$ are both hyperbolic and we have
\[\{f_1^+,f_1^-\}=\{\eta g^n,g^+ g^n\}\sub U\quad\text{ and }\quad\{f_2^+,f_2^-\}=\{\eta g^{-n},g^-g^{-n}\}\sub V\]
which implies that $f_1$ and $f_2$ have no common fixed point even though they are hyperbolic.
This contradiction shows that there is a $\mu\in\rand X$ that lies in $\{g^+,g^-\}f$ for all $f\in G$.
Let $\nu$ be the other element of $\{g^+,g^-\}$.

Since $G$ fixes no element of~$\L(G)$, there is an $f\in G$ with $\mu f\neq\mu$.
Then we have $\nu f=\mu$ and $\nu f^2=\mu f\neq\mu$.
As $\mu\in\{\mu,\nu\}f^2$, we conclude that $\mu f^2=\mu$.
Since $f$ is a homeomorphism on~$\hat{X}$ and $\nu f=\mu f^2$, we have $\mu f=\nu$.
Because of $|\L(G)|\neq 2$, there is a hyperbolic automorphism $f'$ in~$G$ with precisely one fixed point in $\{g^+,g^-\}$, as any two hyperbolic automorphisms have a common fixed point.
If this fixed point is~$\nu$, then we conclude $\mu f'=\mu$ as $\mu\in\{\mu,\nu\}f'$.
By the choice of~$f'$, this is not possible.
So $f'$ fixes $\mu$.
Hence, the automorphism $f'f$ maps $\mu$ to~$\nu$ and $\nu$ to~$\nu f'f\neq \nu f=\mu$.
So $\mu$ does not lie in $\{g^+,g^-\}f'f$.
This contradiction shows that a unique element of~$\L(G)$ is fixed by~$G$ in the situation that each two hyperbolic automorphisms have a common fixed point.

\medskip

Let us consider the remaining case, that is, that there are two hyperbolic elements $g$ and $h$ in~$G$ without common fixed point.
We shall show that there is a $k\geq 1$ such that $g^k$ and $h^k$ satisfy the condition~(\ref{item_W3_1}).
Let $U_1$, $V_1$, $U_2$, and $V_2$ be open neighbourhoods in~$\hat{X}$ of~$g^-$, $g^+$, $h^-$, and $h^+$, respectively, that have pairwise positive distance from each other, such that their closures are disjoint and such that
\[\overline{U_1}\cap\overline{U_2}\cap\overline{V_1}\cap\overline{V_2}\neq\hat{X}.\]
We can find these neighbourhoods similarly as in the proof of Lemma~\ref{lem_UVExistNicely}.
Let $O$ be a non-empty open subset of~$X$ that is disjoint from all four just defined subsets of~$\hat{X}$.
As $\hat{X}$ is contractive, there is an $n_0\geq 1$ with
\[(\hat{X}\sm U_1)g^n\sub V_1\quad\text{ and }\quad (\hat{X}\sm V_1)g^{-n}\sub U_1\]
as well as
\[(\hat{X}\sm U_2)h^n\sub V_2\quad\text{ and }\quad (\hat{X}\sm V_2)h^{-n}\sub U_2\]
for all $n\geq n_0$.
Set $f_1:=g^{n_0}$ and $f_2:=h^{n_0}$.
We shall show that $f_1$ and $f_2$ freely generate $F:=\langle f_1,f_2\rangle$ and that this group acts discontinuously on~$X$.

Let $W_1:= U_1\cup V_1$ and $W_2:= U_2\cup V_2$.
We consider any non-trivial word $f$ in~$F$ and show that it does not represent the trivial element.
In order to show this, we distinguish several cases.

Let us first assume that $f=f_1^{k_1}f_2^{\ell_1}\ldots f_1^{k_m}f_2^{\ell_m}$ for integers $k_i,\ell_i\neq 0$.
Then we conclude
\[(\hat{X}\sm W_1)f=(\hat{X}\sm W_1)f_1^{k_1}f_2^{\ell_1}\ldots f_1^{k_m}f_2^{\ell_m}\]
\[
\sub W_1f_2^{\ell_1}\ldots f_1^{k_m}f_2^{\ell_m}\sub\ldots\sub W_1f_2^{\ell_m}\sub W_2\]
and analogously $(\hat{X}\sm W_2)f\inv\sub W_1$.
We apply Lemma~\ref{lem_CertainHypExists} and thus, know that $f$ is a hyperbolic automorphism with $f^+\in W_2$ and $f^-\in W_1$, in particular, $f$ is not the identity element and $Of\cap O\sub W_2\cap O=\es$.

If $f=f_2^{k_1}f_1^{\ell_1}\ldots f_2^{k_m}f_1^{\ell_m}$, then an analogue proof shows that $f$ is hyperbolic with $f^+\in W_1$ and $f^-\in W_2$ and $Of\cap O=\es$.

Let us now assume that $f=f_1^{k_1}f_2^{\ell_1}\ldots f_1^{k_m}$ for integers $k_i,\ell_i\neq 0$.
First, we notice that $f=f_1^{-k_m}f'f_1^{k_m}$ with $f'=f_1^{k_m+k_1}f_2^{\ell_1}\ldots f_2^{\ell_{m-1}}$ and hence, that $f$~-- as a conjugate of a hyperbolic automorphism~-- is also hyperbolic.
Then we conclude as above that $\{f^+,f^-\}=\{f'^+,f'^-\}\sub W_1$ and $Of\cap O=\es$.

Again, the case $f=f_2^{k_1}f_1^{\ell_1}\ldots f_2^{k_m}$ is analogue to the previous one.
Thus, we have shown that the hyperbolic automorphisms $f_1$ and $f_2$ satisfy the condition~(\ref{item_W3_1}).
\end{proof}

The hyperbolic limit set is \emph{bilaterally dense} in~$\L(G)$ if $\H(G)$ is not empty and if for any two disjoint non-empty open sets $A,B\sub \L(G)$ there is a hyperbolic element of~$G$ such that its direction lies in~$A$ and the direction of its inverse lies in~$B$.

\begin{thm}\label{thm_P5}
Let $\hat{X}$ be a contractive $G$-completion of a metric space~$X$.
The following statements are equivalent.
\begin{enumerate}[{\em (a)}]
\item\label{item_P5_1} The hyperbolic limit set of~$G$ is bilaterally dense in~$\L(G)$.
\item\label{item_P5_2} Either $|\L(G)|=2$ or $G$ contains two hyperbolic elements that have no common fixed point.
\end{enumerate}
\end{thm}

Note that Theorem~\ref{thm_P5}\,(\ref{item_P5_2}) says that in Theorem~\ref{thm_W3} either (\ref{item_W3_4}) or~(\ref{item_W3_1}) holds.

\begin{proof}[Proof of Theorem~\ref{thm_P5}]
Let us assume that (\ref{item_P5_1}) holds and that $|\L(G)|\neq 2$.
As $\H(G)\neq \es$ by the definition of bilateral denseness, we know that $\L(G)$ and $\H(G)$ are infinite according to Theorem~\ref{thm_W2}\,(\ref{item_W2_1}) and~(\ref{item_W2_3}).
As $\hat{X}$ is Hausdorff, we may take four pairwise disjoint open subsets $V_1,\ldots,V_4$ of~$\rand X$ and conclude that there are two hyperbolic elements $g,h$ in~$G$ with $g^+\in V_1$, $g^-\in V_2$, $h^+\in V_3$, and $h^-\in V_4$.
Obviously, these two hyperbolic automorphisms have no common fixed point.

\medskip

To show the converse, let us assume that (\ref{item_P5_2}) holds.
For every open neighbourhood $Y$ in~$\L(G)$ of any element $\eta\in\L(G)$, there is a neighbourhood $Y'$ in~$\hat{X}$ with $Y'\cap\L(G)\sub Y$ as $\hat{X}$ is regular.
Thus, we may take disjoint non-empty open subsets $A$ and $B$ of~$\hat{X}$ with $A':=A\cap\L(G)\neq\es$ and $B':= B\cap\L(G)\neq\es$ and just have to show that there is a hyperbolic element $f$ in~$G$ with $f^+\in A'$ and $f^-\in B'$.

If $|\L(G)|=2$, then each of the two sets $A'$ and $B'$ consists of precisely one point and according to Theorem~\ref{thm_W2}\,(\ref{item_W2_2}) there is a hyperbolic element $f$ in~$G$ with $f^+\in A'$.
This implies $f^-\in B'$.
Hence, we may assume that $G$ contains two hyperbolic elements without common fixed point.

Let $\eta\in A'$ and $\mu\in B'$, let $U$ be an open neighbourhood of~$\eta$ with $\overline{U}\sub A$, and let $V$ be an open neighbourhood of~$\mu$ with $\overline{V}\sub B$ such that $d(U,V)>0$, $\overline{U}\cap\overline{V}=\es$, and $\overline{U}\cup\overline{V}\neq\hat{X}$.
For the existence of~$U$ and~$V$, we refer again to the proof of Lemma~\ref{lem_UVExistNicely}.
Our next aim is to show that there are hyperbolic elements $g,h\in G$ with $g^+,g^-\in U$ and $h^+,h^-\in V$.
As $\H(G)$ is dense in~$\L(G)$, we find a hyperbolic automorphism $a$ in~$G$ with $a^+\in U$.
Since there are two hyperbolic elements in~$G$ without common fixed point, we find a hyperbolic automorphism $b$ that fixes neither~$a^+$ nor~$a^-$.
Applying contractivity to open neighbourhoods $U'$ and~$V'$ of~$a^+$ and $a^-$, respectively, with $U'\sub U$ we obtain an $n\in\nat$ with $b^+a^n\in U$ and $b^-a^n\in U$.
Let $g=a^{-n}ba^n$.
Then $g$ is hyperbolic as it is conjugated to a hyperbolic automorphism and for every $x\in\hat{X}\sm U$ we have
\[xg^m=xa^{-n}b^ma^n\to b^+a^n=g^+\text{ for }m\to\infty\]
and
\[xg^{-m}=xa^{-n}b^{-m}a^n\to b^-a^n=g^-\text{ for }m\to\infty.\]
Thus, $g^+$ and $g^-$ lie in~$U$.
Analogously, we find a hyperbolic element $h$ of~$G$ with $h^+,h^-\in V$.

By contractivity, there is an $m\in\nat$ with $xg^m\in U$ and $xg^{-m}\in U$ for all $x\in \hat{X}\sm U$ as well as $xh^m\in V$ and $xh^{-m}\in V$ for all $x\in\hat{X}\sm V$.
Let $f=h^mg^m$.
Then we conclude
\[xf=xh^mg^m\in Vg^m\sub U\]
for all $x\in \hat{X}\sm V$ and
\[xf\inv=xg^{-m}h^{-m}\in Uh^{-m}\sub V\]
for all $x\in \hat{X}\sm U$.
As $d(U,V)>0$, $\overline{U}\cap\overline{V}=\es$, and $\overline{U}\cup\overline{V}\neq\hat{X}$, Lemma~\ref{lem_CertainHypExists} implies that $f$ is hyperbolic with $f^+\in \overline{U}$ and $f^-\in\overline{V}$ as desired.
\end{proof}

\begin{thm}\label{thm_P7}
Let $\hat{X}$ be a contractive $G$-completion of a metric space $X$ and such that $\L(G)$ is infinite.
Then $\L(G)$ is a perfect set.

Furthermore, the following statements are equivalent.
\begin{enumerate}[{\em (a)}]
\item\label{item_P7_2} The set $\{(g^+,g^-)\mid g\in G,\, g\text{ is hyperbolic}\}$ is dense in $\L(G)\times \L(G)$.
\item\label{item_P7_1} The hyperbolic limit set of~$G$ is bilaterally dense in~$\L(G)$.
\item\label{item_P7_3} There are two hyperbolic elements in~$G$ that have no common fixed point.
\end{enumerate}
\end{thm}

\begin{proof}
To show that $\L(G)$ is perfect, we have to show that $\L(G)$ contains no isolated point.
Let us suppose that $\eta\in\L(G)$ is isolated.
As $\H(G)$ is dense in~$\L(G)$ according to Theorem~\ref{thm_W2}\,(\ref{item_W2_2}), we find a hyperbolic element $g\in G$ with $g^+=\eta$.
Let $\mu\in\L(G)$ with $\mu g\neq\mu$.
This limit point exists as $\L(G)$ is infinite.
Since $g$ is hyperbolic and $\hat{X}$ is contractive, the sequence $(\mu g^i)_{i\in\nat}$ converges to~$g^+$ but none of its elements is~$g^+$.
Hence, $g^+$ cannot be isolated in~$\L(G)$.

\medskip

For the additional statement, we note that (\ref{item_P7_1}) is oviously a direct consequence of~(\ref{item_P7_2}).
The fact that $\L(G)$ is perfect implies the inverse direction and the equivalence of (\ref{item_P7_1}) and (\ref{item_P7_3}) is a special case of Theorem~\ref{thm_P5}.
\end{proof}

\section{Hyperbolic spaces}\label{sec_Hyp}

In this section, we consider hyperbolic spaces that are not necessarily proper\footnote{A metric space is \emph{proper} if all closed balls of finite diameter are compact.}
 but \emph{geodesic}, that is for every two points $x,y$ there is an isometric image of $[0,d(x,y)]$ joining $x$ and~$y$.
We shall show that the geodesic hyperbolic spaces with their hyperbolic boundary are contractive $G$-completions and hence, that the theorems of Section~\ref{sec_Def} are true for them.
To obtain an overview what basic properties of geodesic hyperbolic spaces are known, we refer to~\cite{BS-Elements,V-GromovHyp} and for an introduction to proper geodesic hyperbolic spaces, we refer to~\cite{ABCFLMSS,CoornDelPapa,SymbolicDynamics,GhHaSur, gromov,KB-BoundaryHypGroup}.
Since we deal with spaces that are not necessarily proper, we will cite from the first list and mainly from~\cite{BS-Elements}.
Let us briefly recall the main definitions for hyperbolic spaces.

\medskip

Let $X$ be a metric space.
The \emph{Gromov-product} $(x,y)_o$ of $x,y\in X$ with respect to the \emph{base-point} $o\in X$ is defined as follows:
\[(x,y)_o:={\textstyle \frac{1}{2}}d(o,x)+d(o,y)-d(x,y).\]
For $\delta\geq 0$, the space $X$ is \emph{$\delta$-hyperbolic} if for given base-point $o\in X$ we have
\[(x,y)_o\geq \min\{(x,z)_o, (y,z)_o\}-\delta\]
for all $x,y,z\in X$.
A space is \emph{hyperbolic} if it is $\delta$-hyperbolic for some $\delta\geq 0$.

It is easy to show that the definition of being hyperbolic does not depend on~$o$, that is, if the space is $\delta$-hyperbolic with respect to~$o\in X$, then for $o'\in X$ there exists $\delta'\geq 0$ such that $X$ is $\delta'$-hyperbolic with respect to~$o'$.

To define the completion $\hat{X}$ of a geodesic hyperbolic space $X$, we define a further metric on~$X$.
For this, let $\varepsilon>0$ with $\varepsilon'=\exp(\varepsilon\delta)-1<\sqrt{2}-1$.
For $x,y\in X$, let
\begin{center}{
\begin{tabular}{ll}
$\varrho_\varepsilon(x,y)=\exp(-\varepsilon(x,y)_o) $&if $x\neq y,$\\
$\varrho_\varepsilon(x,y)=0$&otherwise.
\end{tabular}
}\end{center}
Then
\[
d_\varepsilon(x,y)=\inf\{\sum_{i=1}^{n-1} \varrho_\varepsilon(x_i,x_{i+1})\mid x_i\in X,\ x_1=x,\ x_n=y\}
\]
for all $x,y\in X$ defines a metric on~$X$ with
\begin{equation}\label{eq_VisualMetric}
(1-2\varepsilon')\varrho_\varepsilon(x,y)\leq d_\varepsilon(x,y)\leq \varrho_\varepsilon(x,y)
\end{equation}
for all $x,y\in X$, see e.g.~\cite[Theorem 2.2.7]{BS-Elements}.
Let $\hat{X}$ be the completion of the metric space $(X,d_\varepsilon)$ and let $\rand X=\hat{X}\sm X$ be the \emph{hyperbolic boundary} of~$X$.
A subset $S$ of~$X$ \emph{separates} two sets $U,V\sub \hat{X}$ \emph{geodesically} if every geodesic between a point of~$U$ and a point of~$V$ intersects non-trivially with~$S$.

Let $A$ and~$B$ be two subsets of a metric space $Y$.
We say that $A$ lies \emph{$\gamma$-close} to~$B$ for some $\gamma\geq 0$ if $d(a,B)\leq\gamma$ for all $a\in A$.
A \emph{triangle} $xyz$ in a geodesic metric space $Y$ is a union of three geodesics~-- called \emph{sides} of the triangle~--, one between each two of the \emph{vertices} $x$, $y$, and $z$ of the triangle.
A triangle is \emph{$\delta$-thin} if any of its sides lies $\delta$-close to the union of its other two sides.
Due to~\cite[Proposition 2.1.3]{BS-Elements}, every triangle in a geodesic $\delta$-hyperbolic space is $4\delta$-thin.

A useful property of the Gromov-product in geodesic hyperbolic spaces is the following:

\begin{lem}\label{lem_GromProd=Dist}
Let $x,y,z\in X$. Then we have for all geodesics $\pi$ between $y$ and~$z$:
\[
d(x,\pi)-8\delta\leq (y,z)_x\leq d(x,\pi).
\]
\end{lem}

\begin{proof}
Let $t\in\pi$ with $d(x,t)=d(x,\pi)$.
Then the triangle-inequalities 
\[
d(x,t)+d(t,y)\geq d(x,y)\quad \text{and}\quad d(x,t)+d(t,z)\geq d(x,z)
\]
gives the second inequality.
For the first inequality, we notice that there is a $w\in\pi$ such that for some geodesics $\pi_y$ and $\pi_z$ between $x$ and~$y$, and $x$ and~$z$, respectively, we have $d(w,u)=d(w,\pi_y)\leq 4\delta$ and $d(w,v)=d(w,\pi_z)\leq 4\delta$ for some $u\in\pi_y$ and $v\in\pi_z$.
So we have
\[
d(x,y)=d(x,u)+d(u,y)\geq d(x,w)+d(w,y)-8\delta
\]
and
\[
d(x,z)=d(x,v)+d(v,z)\geq d(x,w)+d(w,z)-8\delta
\]
as triangles are $4\delta$-thin. These two inequalities lead to the first one of our assertion.
\end{proof}

We call a map $\varphi:Y\to Z$ between metric spaces \emph{quasi-isometric} if there are $\gamma\geq 1$ and $c\geq 0$ such that
\[
\frac{1}{\gamma}d_Y(y,y')-c\leq d_Z(\varphi(y),\varphi(y'))\leq\gamma d_Y(y,y')+c
\]
for all $y,y'\in Y$.
A~\emph{quasi-geodesic} is the image of a quasi-isometric map $\varphi:[0,r]\to Z$ with $r\in\real_{\geq 0}$ and an \emph{infinite quasi-geodesic} is the image of a quasi-isometric map $\varphi:\real_{\geq 0}\to Z$.

Equipped with these definitions we are able to prove that the hyperbolic completions of geodesic hyperbolic spaces are contractive G-completions.
The following lemma is similar to \cite[Lemme 2.2]{CoornDelPapa}.

\begin{lem}\label{lem_CDP2.2}
Let $X$ be a geodesic hyperbolic space and $g\in\Aut(X)$ with $d(x,xg^2)\geq d(x,xg)+8\delta+\gamma$ for some $\gamma>0$ and $x\in X$.
Then there are two distinct boundary points $\eta$, $\mu$ of~$X$ with $(xg^n)_{n\in\nat}\to\eta$ and $(xg^{-n})_{n\in\nat}\to\mu$.

Furthermore, the map $\ganz\to\{xg^z\mid z\in\ganz\}$, $z\mapsto xg^z$ is quasi-isometric.
\end{lem}

\begin{proof}
Let us first show that the inequalities
\begin{equation}\label{item_CDP2.2.1}
m\,\gamma-\gamma\leq d(x,xg^m)\leq m\, d(x,xg)
\end{equation}
hold for all $m\in\nat$.
The second inequality is obvious by triangle-inequality, so we just have to prove the first one.
Let $m\in\nat$.
Using the quadruple conditions for hyperbolic spaces (cp.\ Section 2.4.1 and Proposition 2.1.3 in~\cite{BS-Elements}) for the points $x$, $xg$, $xg^2$, and $xg^m$, we obtain
\[
d(x,xg^2)+d(xg,xg^m)\leq\max\{d(x,xg)+d(xg^2,xg^m),d(x,xg^m)+d(xg,xg^2)\}+8\delta.
\]
Hence, we have
\begin{align}
\label{item_CDP2.2.2}\max\{d(x,xg^{m-2}),d(x,xg^m)\}&\geq d(x,xg^2)+d(x,xg^{m-1})-d(x,xg)-8\delta\\
&\geq d(x,xg^{m-1})+\gamma.\notag
\end{align}

An easy induction using $d(x,xg^2)\geq d(x,xg)+8\delta+\gamma$ and (\ref{item_CDP2.2.2}) shows
\[
d(x,xg^{n+1})\geq d(x,xg^n)+\gamma
\]
for all $n\in\nat$ and hence, we have $d(x,xg^m)\geq (m-1)\gamma$.

Due to~(\ref{item_CDP2.2.1}), the map $\ganz\to X$, $z\mapsto xg^z$ is quasi-isometric, so we conclude with Theorem 4.4.1 and Proposition 5.2.10 of~\cite{BS-Elements} that $\{xg^n\mid n\in\nat\}$ and $\{xg^{-n}\mid n\in\nat\}$ converge to distinct boundary points.
\end{proof}

\begin{lem}\label{lem_ProdHyp}
Let $X$ be a geodesic hyperbolic space and let $x\in X$.
Let $g,h\in\Aut(X)$ such that $d(x,xg^2)\leq d(x,xg)+8\delta$ and $d(x,xh^2)\leq d(x,xh)+8\delta$ and such that neither $g$ nor~$h$ satisfies the conclusions of Lemma~\ref{lem_CDP2.2}.
If there is a ball $B$ with centre $x$ and radius $R$ such that any geodesic between $x$ and $xgh$ intersects non-trivially with~$Bg$ and if we have $d(B',B'g)>8\delta$ and $d(B'g,B'gh)>8\delta$ for the ball $B'$ with centre $x$ and radius $R+16\delta$, then
\[
d(x,x(gh)^2)> d(x,xgh)+8\delta.
\]
\end{lem}

\begin{proof}
We consider the following points in~$X$: $x$, $xg$, $xh$, $xgh^2$, $xg^2h$, and $x(gh)^2$.
If we can show that $xgh$ lies $16\delta$-close to any geodesic between $x$ and~$x(gh)^2$, then this geodesic must intersect non-trivially with $B'gh$ and we obtain
\[
d(x,x(gh)^2)\geq 2d(x,xgh)- 2(R+16\delta)> d(x,xgh)+8\delta.
\]

\medskip

Let us consider a geodesic between $xg$ and~$xgh^2$.
If it intersects non-trivially with $B'gh$, then we conclude
\[
d(xg,xgh^2)\geq d(xg,xgh)+d(xgh,xgh^2)-2(R+16\delta)+8\delta>d(xg,xgh)+8\delta
\]
and we apply Lemma~\ref{lem_CDP2.2} to obtain a contradiction to our assumptions.
Hence, no such geodesic intersects non-trivially with~$B'gh$.
Similarly, if we consider any geodesic between $x$ and~$xg^2$, then we obtain that it intersects non-trivially with $B'g$, so the same holds for any geodesic between $xh$ and $xg^2h$ with the ball $B'gh$.

Since the triangles are $4\delta$-thin, we obtain that $[xh,xgh^2]$ lies $16\delta$-close to
\[
[xgh^2,xg]\cup [xg,x]\cup [x,x(gh)^2]\cup [x(gh)^2,xg^2h]\cup [xg^2h,xh]
\]
where the brackets denote any geodesic between the two points.
As $[xh,xgh^2]$ intersects non-trivially with~$Bgh$, one of the other five geodesics intersects non-trivially with $B'gh$.
We have already shown that this is neither $[xgh^2,xg]$ nor $[xg^2h,xh]$.
The geodesics $[xg,x]$ and $[x(gh)^2,xg^2h]$ do not intersect non-trivially with $Bgh$, too, since $Bg$ separates $x$ and $xgh$ geodesically and the same is true for~$Bg^2h$ with $x(gh)^2$ and $xgh$.
So $[x,x(gh)^2]$ intersects non-trivially with $B'gh$ and the assertion follows as described above.
\end{proof}

Now we are able to prove that geodesic hyperbolic spaces are contractive $G$-completions.

\begin{prop}
Let $X$ be a geodesic hyperbolic space and $\hat{X}$ the completion of~$X$ with the hyperbolic boundary.
Then $\hat{X}$ is a contractive $\Aut(X)$-completion of~$X$.
\end{prop}

\begin{proof}
By its definition, $\hat{X}$ is a completion of~$X$ and from~\cite[Section~2.2.3]{BS-Elements} we deduce that isometries of~$X$ extend to homeomorphisms of~$\hat{X}$.
As $(\hat{X},d_\varepsilon)$ is a metric space, it is regular.
Let $(g_i)_{i\in\nat}$ be a sequence in~$\Aut(X)$ such that $d(x,xg_i)$ is unbounded for some $x\in X$.
We will show~(\ref{item_CompletionIII}).
Let us consider closed balls $B_i$ with centre $x$ and radius~$i$.
Either, for all $i$, all but finitely many $xg_j$ are not geodesically separated by~$B_i$ or there are a ball $B_i$ and $k,\ell\in\nat$ such that $B_i$ separates $xg_k$ and $xg_\ell$ geodesically and $d(B_ig_k,B_i)>8\delta$ and $d(B_ig_\ell,B_i)>8\delta$.
In the first case, we obtain $(xg_k,xg_\ell)\to\infty$ for $k,\ell\to\infty$ because of Lemma~\ref{lem_GromProd=Dist}, so the sequence converges to some boundary point.
In the second case, either one of~$g_k\inv$ and~$g_\ell$ or due to Lemma~\ref{lem_ProdHyp} the automorphism $g_k\inv g_\ell$ has the desired limit points by Lemma~\ref{lem_CDP2.2}.
This shows~(\ref{item_CompletionIII}).

For the proof of~(\ref{item_CompletionIIStronger}) let $(g_i)_{i\in\nat}$ be a sequence in~$G$ such that for some $x\in X$ and $\eta\in\rand X$ we have $xg_i\to\eta$ for $i\to\infty$ and such that $\{xg_i\inv\mid i\in\nat\}$ has no accumulation point in~$\rand X$.
Notice that the convergence of the sequence $(xg_i)_{i\in\nat}$ implies $d(x,xg_i)\to\infty$ for $i\to\infty$.
Analogously as in the proof of~(\ref{item_CompletionIII}), we find $k,\ell\in\nat$ such that one of the automorphisms $g_k$, $g_\ell$, and $g_kg_\ell$ satisfies the assumption of Lemma~\ref{lem_CDP2.2} and hence fulfills the conclusions of that lemma.
Furthermore, we find for each $k,\ell\in\nat$ further integers $k',\ell'\in\nat$ both larger than $k$ and~$\ell$ such that among $g_{k'}$, $g_{\ell'}$ and $g_{k'}g_{\ell'}$ we find another automorphism that satisfies the conclusions of Lemma~\ref{lem_CDP2.2}.
So we find an infinite sequence $(f_i)_{i\in\nat}$ of such automorphisms: $f_i$~is either some~$g_m$ or some~$g_mg_n$ and for $i\to\infty$ also the indices $m$ and~$n$ grow.
Using this sequence, we shall construct another sequence $(h_i)_{i\in\nat}$ of automorphisms such that each of the two sets $\{xh_i^n\mid n\in\nat\}$ and $\{xh_i^{-n}\mid n\in\nat\}$ has a limit point $\eta_i$ and $\mu_i$, respectively, such that these two limit points are distinct and such that $\eta_i\to\eta$ for $i\to\infty$.
The first of these two properties is also true for each $f_i$ and we will use this for the proof of the property for the $h_i$.

Let us consider the open balls $B_{1/n}(\eta)$.
For $f_n$, there is a constant $\Delta_n$ due to \cite[Theorem 1.3.2]{BS-Elements} such that any geodesic between $xf_n^{-m}$ and $xf_n^m$ lies $\Delta_n$-close to $\{xf_n^j\mid |j|\leq m\}$.
As $xg_i\to\eta$ for $i\to\infty$, we find $i_n$ such that the $\Delta_n$-ball $B$ with centre $xg_{i_n}$ lies completely in $B_{1/n}(\eta)$.
Let $h_n=g_{i_n}\inv f_n g_{i_n}$.
Since $h_n$ is conjugated to~$f_n$, the conclusions of Lemma~\ref{lem_CDP2.2} also hold for~$h_n$.
Let us consider the two sets $Q_1:=\{xf^j_n\mid j\in\nat\}$ and $Q_2:=\{xf^{-j}_n\mid j\in\nat\}$ and, for $i=1,2$, quasi-geodesics $R_i$ that contain all elements of~$Q_i$ and a geodesic between any $xf^j_n$ and~$xf^{j+1}_n$.
The ball $B$ separates any $q_1\in R_1$ from any $q_2\in R_2$ geodesically by its choice.
Thus, one of the two quasi-geodesics, say $R_i$, has distance at least $d(x,B)/2$ to~$x$.
As $R_i$ is quasi-geodesic, it has a limit point $\eta_n\in\rand X$.
Using $4\delta$-thin triangles $z_\ell(xg_{i_n})q_\ell$ for sequences $(z_\ell)_{\ell\in\nat}$ in $B_{1/n}(\eta)$ and $(q_\ell)_{\ell\in\nat}$ in~$R_i$ converging to~$\eta$ and $\eta_n$, respectively, we obtain by Lemma~\ref{lem_GromProd=Dist}
\[
(z_\ell,q_\ell)\geq d(x,\pi_\ell)-8\delta=d(x,B)/2-12\delta,
\]
where $\pi_\ell$ is a geodesic between $z_\ell$ and~$q_\ell$.
So due to (\ref{eq_VisualMetric}), we know that the sequence $(\eta_k)_{k\in\nat}$ converges to~$\eta$.
Notice that we might have to change some $h_i$ in the sequence $(h_i)_{i\in\nat}$ to $h_i\inv$ to obtain precisely the statement of~(\ref{item_CompletionIIStronger}).

For the projectivity property, let $(x_i)_{i\in\nat}$ be a sequence in~$X$ that converges to some $\eta\in\rand X$ and let $(y_i)_{i\in\nat}$ be another sequence in~$X$ such that there is an $M\geq 0$ with $d(x_i,y_i)\leq M$ for all $i\in\nat$.
As $(x_i)_{i\in\nat}$ converges to a boundary point, we have $d(o,x_i)\to\infty$ and thus also $d(o,y_i)\to\infty$.
This implies $(x_i,y_i)\to \infty$, so $d_\varepsilon(x_i,y_i)\to 0$.
Hence, $\hat{X}$ is a projective $G$-completion.

To show contractivity, let $(g_i)_{i\in\nat}$ be a sequence in~$\Aut(X)$ such that for the base point $x\in X$ of the Gromov-product, the sequence $(xg_i)_{i\in\nat}$ converges to~$\eta\in\rand X$ and $(xg_i\inv)_{i\in\nat}$ converges to~$\mu\in\rand X$.
Let $U$ and~$V$ be open neighbourhoods of~$\eta$ and~$\mu$, respectively.
Then there are $\theta>0$ and $n_0\in\nat$ such that
\[
\{xg^m\mid m\geq n_0\}\cup\{\eta\}\sub B_{\theta/3}(xg_n)\quad \text{and}\quad B_\theta(xg_n)\sub U
\]
as well as
\[
\{xg^{-m}\mid m\geq n_0\}\cup\{\mu\}\sub B_{\theta/3}(xg_{-n})\quad \text{and}\quad B_\theta(xg_{-n})\sub V
\]
for all $n\geq n_0$.
Let $y\in X\sm B_{2\theta/3}(xg_{n_0}\inv)$.
Then we have $d_\varepsilon(y,\mu)\geq\theta/3$ and $\exp(-\varepsilon(xg_n\inv,y))\geq d_\varepsilon(xg_n\inv, y)\geq\theta/3$.
We conclude
\begin{align*}
(xg_n,yg_n)&={\textstyle \frac{1}{2}}(d(x,xg_n)+d(x,yg_n)-d(xg_n,yg_n))\\
&={\textstyle \frac{1}{2}} (d(x,xg_n)+d(x,yg_n)-d(x,y))\\
&=d(x,xg_n)-(xg_n\inv,y).
\end{align*}
As $d(x,xg_n)\to\infty$ for $n\to\infty$, we find $n_1\in\nat$ such that we have
\begin{align*}
d_\varepsilon(xg_n,yg_n)&\leq\varrho_\varepsilon(xg_n,yg_n)\\
&=\exp(-\varepsilon d(x,xg_n)+\varepsilon(xg_n\inv,y))\\
&\leq\exp(-\varepsilon d(x,xg_n)-\log(\theta/3))\\
&<\theta/3.
\end{align*}
for all $n\geq n_1$.
So $yg_n$ lies in $B_{\theta/3} (xg_n)\sub U$.
Let $\nu\in\rand{X}\sm V$.
Then we can find a sequence $(y_i)_{i\in\nat}$ in~$X\sm B_{2\theta/3}(xg_{-n})$ that converges to~$\nu$.
Since $y_ig_n\in B_{\theta/3} (xg_n)$, we conclude that $\nu g_n$ lies in $B_{\theta} (xg_n)$.
This shows contractivity and hence, we have shown that $\hat{X}$ is a contractive $\Aut(X)$-completion.
\end{proof}

We directly obtain:

\begin{cor}\label{cor_HypBound}
Let $X$ be a geodesic hyperbolic space and $\hat{X}$ the completion of~$X$ with its hyperbolic boundary.
Then the theorems of Section~\ref{sec_results} hold for~$\hat{X}$.\qed
\end{cor}

\section{Graphs with their ends}\label{sec_App:Graphs}

Contractive $G$-completions are natural generalizations of the contractive $G$-compactifications defined by Woess \cite{W-FixedSets}.
Besides proper geodesic hyperbolic spaces, examples for those contractive $G$-compactifications are locally finite connected graphs $X$ with \emph{vertex ends} as boundary (see~\cite{W-FixedSets}) that are the equivalence classes of \emph{rays} (i.e.\ one-way infinite paths) where two rays are equivalent if and only if they lie eventually in the same component of $X\sm S$ for any finite vertex set~$S$.
A base for the topology on a graph with its vertex ends is given by sets that are open in the distance metric of the graph and by sets $C$ of vertices that have a finite neighbourhood (vertices in $V(G)\sm C$ that are adjacent to some vertex of~$C$) and such that some ray lies in~$C$.
In this latter situation, the set $C$ is a neighbourhood of all vertex ends that have a ray which lies in~$C$.

For our theorems, we dropped the hypothesis on~$X$ being a proper metric space, that is, we do not require the graphs to have finite degrees.
Thus, the canonical guess would be to ask if arbitrary connected graphs $X$ with their vertex ends are examples of $G$-completions.
Unfortunately, this is not the case: the first obstacle is that such a space is not projective and the second is that the uniform convergence property of the contractivity does not hold for the space.
We give an example for these two obstacles:

\begin{exam}\label{exam_vertexEnds}
Let $X$ be a graph such that every vertex is a cut vertex and lies in $\lambda$ blocks each of which is a copy of the complete graph on $\kappa$ vertices, where $\kappa$ and $\lambda$ are infinite cardinals.
These graphs have a large symmetry group: its automorphisms do not only act transitively on the graph.
Indeed, the graphs are \emph{distance-transitive} graphs\footnote{A graph is called \emph{distance-transitive} if, for each $k\in \nat$, its automorphisms act transitively on those pairs of vertices that have distance $k$ to each other.}, cp.~\cite{HP-Transitivity,DistanceTransitive}.

Considering the completion $\hat{X}$ of $X$ with its vertex ends, any two rays in distinct blocks have bounded distance to each other but they lie in distinct vertex ends.
Thus, $\hat{X}$ is not projective.

To see that also the second part of the definition of contractivity -- the uniform converging property -- does not hold, let $Y$ be a block in~$X$ and $C$ be a component of $X-y$ for a vertex $y\in Y$ with $C\cap Y=\es$.
Let $(y_i)_{i\in\nat}$ be a sequence in~$Y$ such that its elements are pairwise distinct and also all distinct from $y$ and such that $Y\sm\{y_i\mid i\in\nat\}$ is infinite.
Let $(C_i)_{i\in\nat}$ be a sequence of components of $X\sm\{y_i\}$ with $C_i\cap Y=\es$.
Then there is an automorphism $g$ of~$X$ with $C_i g=C_{i+1}$ that fixes $C$ pointwise.
Thus, we have $xg^i\to\eta$ for $i\to\infty$ and for every $x\in C_1$, where $\eta$ is the end that contains all rays in~$Y$, and also $xg^{-i}\to\eta$ for $i\to\infty$.
There is a neighbourhood $U$ of~$\eta$ that intersects with $C$ trivially.
Hence, $xg^n$ has to converge to~$\eta$ for every $x\in C$ if the uniform convergence property holds, but $g$ fixes~$C$ pointwise, so we have $xg^n=x$.
This shows that also uniform convergence fails for~$X$ and it finishes Example~\ref{exam_vertexEnds}.
\end{exam}

But nevertheless, at least Theorem~\ref{thm_W1} and Theorem~\ref{thm_W3} are true for connected graphs with their vertex ends as boundary, see Halin~\cite{H-AutoAndEndo} and Jung~\cite{J-FiniteFixedSets}.
Although the vertex ends fail to make $\hat{X}$ a $G$-completion in general, there is on one side a natural subclass of the ends and on the other side another notion of ends, the metric ends as defined by Kr\"on in~\cite{Kroen-EndCompact} (see also Kr\"on and M\"oller \cite{KroenMoeller-MetricEnds,KM-QuasiIsometries}), which our situation fits to.

\medskip

We call a ray a \emph{local ray} if there is a vertex set of finite diameter that contains infinitely many vertices of the ray.
As we have seen in Example~\ref{exam_vertexEnds}, two distinct ends each of which contains a local ray is an obstruction for any completion of a graph to be projective and any end that contains a local ray might be an obstacle for the uniform convergence property in the definition of contractivity.
This motivates us to consider only those ends for the contractive $G$-completion that do not contain any local ray.
And indeed, we obtain the following result:

\begin{thm}\label{thm_VertexEnds}
Let $X$ be a connected graph and $\hat{X}$ the completion of~$X$ with all those vertex ends of~$X$ that do not contain any local ray.
Then $\hat{X}$ is a contractive $\Aut(X)$-completion and the theorems of Section~\ref{sec_results} hold for~$\hat{X}$.
\end{thm}

The proof of Theorem~\ref{thm_VertexEnds} is similar to the one of Theorem~\ref{thm_MetricEnds} but uses finite vertex sets instead of vertex sets of finite diameter.
Notice that Spr\"ussel~\cite[Theorem 2.2]{S-EndsNormal} showed that graphs with their ends form a normal topological space.
We omit the proof of Theorem~\ref{thm_VertexEnds} and prove the results for connected graphs with their metric ends instead.

\medskip

A ray in a graph $X$ is a \emph{metric ray} if it eventually lies outside every ball of finite diameter.
So a ray is a metric ray if and only if it is not a local ray.
Two metric rays are \emph{equivalent} if they eventually lie in the same component of $X\sm S$ for any vertex set $S$ of finite diameter.
This is an equivalence relation and its equivalence classes are the \emph{metric ends} of~$X$.
A \emph{metric double ray} is a \emph{double ray} (i.e.\ a two-way infinite path) such that no ball of finite diameter contains infinitely many of its vertices.
So any subray of a metric double ray is a metric ray.
Let us define a base for the topology on a graph with its metric ends:
it consists of all those sets that are open in the distance metric of the graph and of all those sets $C$ of vertices that have a neighbourhood of finite diameter and such that some metric ray lies in~$C$~-- in this situation the set $C$ is a neighbourhood of all metric ends that have a metric ray which lies in~$C$.

To prove that a connected graph $X$ with its metric ends is an $\Aut(X)$-completion, we need a result due to Kr\"on and M\"oller~\cite{KroenMoeller-MetricEnds}, which is (for a connected graph) a stronger version of Lemma~\ref{lem_CertainHypExists}.

\begin{thm}\label{thm_KM2_12}\cite[Theorem~2.12]{KroenMoeller-MetricEnds}
Let $X$ be a connected graph and $g\in\Aut(X)$.
If there is a non-empty vertex set $S$ of finite diameter, a component $C$ of~$X\sm S$ and an $n\in\nat$ with $(S\cup C)g^n\sub C$, then there is a metric double ray $L$ and an $m\in\nat$ such that $g^m$ acts as a non-trivial translation on~$L$.\qed
\end{thm}

\begin{thm}\label{thm_MetricEnds}
Let $X$ be a connected graph and $\hat{X}$ be $X$ with its metric ends.
Then $\hat{X}$ is a contractive $\Aut(X)$-completion of~$X$ and the theorems of Section~\ref{sec_results} hold for~$\hat{X}$.
\end{thm}

\begin{proof}
First, we mention that $\hat{X}$ is Hausdorff and regular, cp.\ \cite[Theorem~4]{Kroen-EndCompact} and that the canonical extensions of automorphisms of~$X$ are homeomorphisms of~$\hat{X}$, cp.\ \cite[Theorem~6]{Kroen-EndCompact}.
Furthermore, $X$ is open and dense in~$\hat{X}$.
Thus, it remains to prove (\ref{item_CompletionIII}) and~(\ref{item_CompletionIIStronger}) for $G=\Aut(X)$ and then that the $G$-completion is contractive.

We note that the condition for $\hat{X}$ being projective is -- as a direct consequence of the definition of metric ends -- valid even though we have not proved yet that $\hat{X}$ is a $G$-completion.
But we may use the property during the remainder of the proof.

\medskip

To prove (\ref{item_CompletionIII}), let $(g_j)_{j\in\nat}$ be a sequence in~$G$ with $d(x,xg_j)\to\infty$ for $j\to\infty$.
Let $B_i$ be the ball with centre $x$ and radius $i$.
Either there is for each $i$ precisely one component of~$X\sm B_i$ that contains all but finitely many vertices of $\{xg_j\mid j\in\nat\}$ or there are two components $C_1,C_2$ of $X\sm B_i$ and $k,\ell\in\nat$ with $B_ig_k\sub C_1$ and $B_ig_\ell\sub C_2$ as well as with $d(B_ig_k,B_i)\geq 2$ and $d(B_ig_\ell,B_i)\geq 2$.
In the first case, those components $D_i$ that contain all but finitely many of the vertices of $\{xg_j\mid j\in\nat\}$ define a unique metric end $\eta$ as the radii of the balls $B_i$ increase strictly: take the unique element in $\bigcap_{i\in\nat}\overline{D}_i$.
As the sequence $(xg_j)_{j\in\nat}$ eventually lies in each of these components, the sequence must have $\eta$ as an accumulation point.

Thus, we may assume that there are two distinct components $C_1,C_2$ of $X\sm B_i$ and $k,\ell\in\nat$ with $B_ig_k\sub C_1$ and $B_ig_\ell\sub C_2$ and with $d(B_ig_k,B_i)\geq 2$ and $d(B_ig_\ell,B_i)\geq 2$.
If either $g_k$ or $g_\ell$ satisfies the assumptions of Theorem~\ref{thm_KM2_12}, then there is a vertex $z$ on the metric double ray $L$ of the conclusion of Theorem~\ref{thm_KM2_12} such that the set $\{zg_j^n\mid n\in\ganz\}$, for either $j=k$ or $j=\ell$, has the metric ends to which every subray of $L$ converges as accumulation points.
By projectivity, we conclude that each of the two sets $\{xg_j^n\mid n\in\nat\}$ and $\{xg_j^{-n}\mid n\in\nat\}$ has an accumulation point in~$\rand X$.
So we assume that neither $g_k$ nor $g_\ell$ satisfies the assumptions of Theorem~\ref{thm_KM2_12}.
This implies that $B_ig_k^2$ must lie in the same component of $X\sm B_ig_k$ in which $B_i$ lies.
Analogously, $B_ig_\ell^2$ lies together with $B_i$ in a component of $X\sm B_ig_\ell$.
Let us consider the automorphism $g:=g_k\inv g_\ell$.
Let $y\in C_1$ with $d(y,B_ig_k)<d(B_i,B_ig_k)=d(B_ig_k^2,B_ig_k)$.
The vertices $yg_k\inv$ and $x$ must lies in the same component of~$X\sm B_ig_k$.
Hence, $x$ and $yg$ do not lie in the same component of $X\sm B_ig_\ell$ and the same is true for~$x$ and~$xg$.
This implies for the component $C$ of $X\sm B_ig_k$ that contains~$x$, that we have $(B_ig_k\cup C)g\sub C$.
According to Theorem~\ref{thm_KM2_12}, there is a metric end that is a limit point of~$\{xg^j\mid j\in\nat\}$ and the same holds for $\{zg^{-j}\mid j\in\nat\}$.
This finishes the proof of~(\ref{item_CompletionIII}).

For the proof of~(\ref{item_CompletionIIStronger}), let $x\in X$ and let $(g_k)_{k\in\nat}$ be a sequence in~$G$ with $xg_k\to\eta$ for $k\to\infty$ for some $\eta\in\rand X$ such that $\{xg_k\inv\mid k\in\nat\}$ has no accumulation point.
So there is an $i_0$ such that we find for all $i\geq i_0$ a ball $B_i$ of radius~$i$ and centre $x$ so that all but finitely many of the balls $B_ig_k$ lie in the same component $C_i$ of~$X\sm B_i$ and furthermore, all but finitely many of the balls $B_ig_k\inv$ lie outside~$C_i$.
If we find infinitely many $g_k$ that satisfy the assumptions of~Theorem~\ref{thm_KM2_12}, then the sets $\{xg_k^n\mid n\in\nat\}$ and $\{xg_k^{-n}\mid n\in\nat\}$ have distinct limit points $\eta_1$ and~$\eta_2$ in the set of metric ends and we find a sequence in the limit points of the sets $\{xg_j^n\mid n\in\nat\}$ that converges to~$\eta$ since for all $k$ with $B_ig_k\sub C_i$ one of the two limit points $\eta_1$ and~$\eta_2$ lies in~$C_i$, that is, contains a metric ray inside~$C_i$.
If we do not find these infinitely many $g_k$, then let $k$ be such that $B_ig_k$ lies in~$C_i$ and let $\ell$ be such that $B_ig_\ell\inv$ lies in a component of~$X\sm B_i$ distinct from~$C_i$ and such that $d(B_i,B_ig_k)>2$ and $d(B_i,B_ig_\ell\inv)>2$.
As in the proof of~(\ref{item_CompletionIII}), the automorphism $g_\ell g_k$ satisfies the assumptions of Theorem~\ref{thm_KM2_12} and, as we can choose $k$ among infinitely many natural numbers, we obtain our sequence of limit points of the sets $\{xg_\ell g_j^n\mid n\in\nat\}$ that converges to~$\eta$ similarly to the previous case.
This shows~(\ref{item_CompletionIIStronger}).

\medskip

Let us now prove that $\hat{X}$ is contractive.
We have already seen that $\hat{X}$ is projective.
So let $\sequ{g}$ be a sequence in~$G$ with $xg_i\to\eta$ and $xg_i\inv\to\mu$ for $i\to\infty$, some $x\in X$ and metric ends $\eta$ and~$\mu$.
Let $U$ be a neighbourhood of~$\eta$ and $V$ be a neighbourhood of~$\mu$.
We may assume that there are vertex sets $S_U$ and $S_V$ of finite diameter such that $U$ is a component of~$X\sm S_U$ and $V$ is a component of~$X\sm S_V$.
As $xg_n\to\eta$, there is an $n_1\in\nat$ such that $S_Vg_n$ lies in the same component of~$X\sm S_U$ as~$\eta$ and such that $d(xg_n,S_U)> d(x,S_V)+\diam(S_V)$ for all $n\geq n_1$.
Then we have $S_Vg_n\sub U$ and in particular $S_Vg_n\cap S_U=\es$ for all $n\geq n_1$.
Similarly, we find $n_2\in\nat$ such that $S_Ug\inv_n$ lies in the same component of~$X\sm S_V$ as~$\eta$ and such that $d(xg\inv_n,S_V)> d(x,S_U)+\diam(S_U)$ for all $n\geq n_2$.
Again, we have $S_Ug\inv_n\sub V$ and $S_Ug\inv_n\cap S_V=\es$ for all $n\geq n_2$ and hence also $S_U\sub Vg_n$.
Let $n_0:=\max\{n_1,n_2\}$, $n\geq n_0$ and $y\in \hat{X}\sm V$.
As $S_U\sub Vg_n$ and $S_V$ separates $y$ and~$S_Ug_n\inv$, the vertex $yg_n$ must lie outside the component of $X\sm (S_Vg_n)$ that contains $S_U$.
Since it is cannot be separated from $\eta$ by~$S_U$, we have $yg_n\in U$.
This shows that $(yg_n)_{n\in\nat}$ converges uniformly to~$\eta$ outside every neighbourhood of all accumulation points of $\{xg\inv_i\mid i\in\nat\}$ in~$\hat{X}$.
\end{proof}

In the case of locally finite graphs with their vertex ends as boundary, a parabolic automorphism $g$ has the additional property that the sequence $(xg^i)_{i\in\nat}$ converges to the unique fixed end for any vertex~$x$.
This is not true in the case of arbitrary graphs with their metric ends as boundary:
Kr\"on and M\"oller \cite[Example 3.16]{KroenMoeller-MetricEnds} constructed a graph with precisely one metric end and an automorphism that fixes no bounded vertex set but leaves a double ray invariant that is neither bounded nor a metric double ray.
This implies that for any vertex $x$ on that double ray, its orbit is unbounded but there is a vertex set of finite diameter that contains infinitely many of the vertices in its orbit.
This shows that, for contractive $G$-compactifications, an analogous converging property as for hyperbolic automorphisms does not hold in the case of parabolic automorphisms.

\section{Concluding remarks}

Apart from the general investigation of groups acting on proper geodesic hyperbolic spaces or on locally finite graphs, there are several more detailed investigations most of which take either Theorem~\ref{thm_W3}\,(\ref{item_W3_3}) or Theorem~\ref{thm_W3}\,(\ref{item_W3_1}) as starting point and investigate these situations in more detail: M\"oller~\cite{EoGII} showed that locally finite graphs with infinitely many ends for which a group of automorphisms acts transitively on the graph but fixes an end are quasi-isometric to trees.
The same result was obtained in~\cite{EndTransitivity} for arbitrary graphs with infinitely many ends.
Caprace et al.~\cite{CCMT-AmenableHyperbolic} showed an analogous result for locally finite hyperbolic graphs where the fixed end is replaced by a fixed hyperbolic boundary point (the planar situation was settled earlier in~\cite{GH-PlanarHyperbolic}).

In~\cite{KM-FreeGroups}, Kr\"on and M\"oller started with the situation of Theorem~\ref{thm_W3}\,(\ref{item_W3_1}) and showed that if a group acts on a connected graph such that no vertex end is fixed by the group, then the group has a free subgroup containing (except for the trivial element) only hyperbolic automorphisms and the directions of these hyperbolic automorphisms are dense in the set of all limit points of the group.
In the same paper, they also mentioned that an analogous proof holds for metric ends instead of vertex ends.
The analogous theorem for proper geodesic hyperbolic spaces is proved in~\cite{HM-FreeGroupsHyp}.

\providecommand{\bysame}{\leavevmode\hbox to3em{\hrulefill}\thinspace}
\providecommand{\MR}{\relax\ifhmode\unskip\space\fi MR }
\providecommand{\MRhref}[2]{%
  \href{http://www.ams.org/mathscinet-getitem?mr=#1}{#2}
}
\providecommand{\href}[2]{#2}

\end{document}